\RequirePackage{fix-cm}
\documentclass[smallextended]{svjour3}       
\smartqed  
\usepackage{graphicx}
\usepackage{mathptmx}      
\usepackage{latexsym}
\usepackage{amsmath}
\usepackage{amssymb}

\newcommand{\R}{\mathbb R}
\newcommand{\N}{\mathbb N}

\newcommand{\X}{\mathbb X}
\newcommand{\Y}{\mathbb Y}
\newcommand{\Uball}{{\mathbb B}}
\newcommand{\Usfer}{{\mathbb S}}

\newcommand{\dom}{{\rm dom}\, }
\newcommand{\graph}{{\rm graph}\,}

\newcommand{\nullv}{\mathbf{0}}

\newcommand{\cl}{{\rm cl}\, }
\newcommand{\conv}{{\rm conv}\, }

\newcommand{\inte}{{\rm int}\, }

\newcommand{\stardif}{\hbox{${*\over {}}$}}

\newcommand{\Lin}{{\rm L}}

\newcommand{\VP}{{\rm VOP}}
\newcommand{\PSV}{{\rm PSV}}
\newcommand{\Solv}{{\mathcal S}}
\newcommand{\IESol}{{\rm IE}}
\newcommand{\ival}{{\rm val}}

\newcommand{\parord}{\le_{{}_C}}
\newcommand{\rparord}{\ge_{{}_C}}
\newcommand{\Cmin}{C\hbox{-}\min}
\newcommand{\FRf}{F_{\mathcal{R},f}}

\newcommand{\calmup}{\overline{\rm clm}\,}
\newcommand{\Fder}{{\rm D}}

\newcommand{\VOP}[1]{{\rm VOP}#1\,}
\newcommand{\stsl}[1]{|\nabla #1|}
\newcommand{\psostsl}[1]{\overline{|\nabla_x #1|}{}^>}

\newcommand{\Diff}[1]{\mathcal{D}(#1)}

\newcommand{\ball}[2]{{\rm B}\left(#1, #2\right)}
\newcommand{\calm}[2]{{\rm clm}\, #1(#2)}
\newcommand{\ucalm}[2]{\overline{\rm clm}\, #1(#2)}

\newcommand{\dist}[2]{{\rm dist}\left(#1,#2\right)}
\newcommand{\exc}[2]{{\rm exc}(#1,#2)}
\newcommand{\incr}[2]{{\rm inc}(#1;#2)}

\newcommand{\Bder}[2]{{\rm D}_B#1(#2)}

\newcommand{\Lipusc}[2]{{\rm Lipusc}\,#1(#2)}

\newcommand{\Liplsc}[3]{{\rm Liplsc}\,#1(#2,#3)}

\begin{document}



\title{A condition for the stability of ideal efficient solutions
in parametric vector optimization via set-valued inclusions
}

\titlerunning{Stability of ideal efficient solutions via set-valued inclusions}        

\author{Amos Uderzo}


\institute{A. Uderzo \at
              Department of Mathematics and Applications \\
              University of Milano-Bicocca \\
              Via R. Cozzi, 55 - 20125 Milano (Italy) \\
              \email{amos.uderzo@unimib.it}
}

\date{Received: date \today / Accepted: date}

\maketitle

\begin{abstract}
In present paper, an analysis of the stability
behaviour of ideal efficient solutions to parametric
vector optimization problems is conducted.
A sufficient condition for the existence of ideal efficient
solutions to locally perturbed problems and their nearness to a
given reference value is provided
by refining recent results on the stability theory of
parameterized set-valued inclusions.
More precisely, the Lipschitz lower semicontinuity property
of the solution mapping is established, with an estimate
of the related modulus.
A notable consequence of this fact is the calmness behaviour of the
ideal value mapping associated to the parametric class of
vector optimization problems.
Within such an analysis, a refinement of a recent existence result
specific for ideal efficient solutions to unperturbed problem is also
discussed.

\end{abstract}


%
%


\section{Introduction and problem statement}

Vector optimization offers a sophisticated and effective theoretical
apparatus for supporting decision processes in the presence of multiple
conflicting criteria.
A peculiar feature of vector optimization is that, in a context
of partial ordering, there are different concepts of solutions,
reflecting different viewpoints and priorities  of the decision
maker.
Among the basic and mainly investigated solution concepts, ideal
efficient solutions are the strongest ones, whose definition
appears very close to the natural definition of solution for
scalar optimization problems. In its global form, an ideal efficient
solution in fact captures the possibility of comparison with
any other admissible choice and, in doing so, it guarantees
better performances with respect to each among the multiple criteria
to be considered.
A drawback of such a concept is that the geometry of ideal efficiency
is very delicate, so that the existence of ideal efficient
solutions can hardly take place in many problems.
For this reason, in the rare circumstances when they
do exist, it becomes important to understand upon which conditions
their existence can be preserved in the presence of data
perturbations and, if this happens, how and how much
they change.
Whereas for the stability analysis of weak efficient and efficient
solutions to vector optimization problems a well-developed
literature can be found (see, among the others,
\cite{Bedn94,CaLoMoPa20,CheCra94,ChHuYa09,ChHuYa10,CraLuu00,SaNaTa85,Tamm94,Tani88}),
a stability analysis specific for ideal efficient solution
seems to be still largely unexplored.
The present paper describes an attempt to address this question.

Consider the following parametric optimization problem
$$
   \Cmin\, f(p,x) \quad\hbox{ subject to }\quad
   x\in\mathcal{R}(p), \leqno (\VOP{p})
$$
where $f:P\times\X\longrightarrow\Y$ is a mapping representing
the vector objective function, $C\subset\Y$ is a nontrivial (i.e. $C\ne\{\nullv\}$)
closed, pointed, convex cone, inducing the partial order relation $\parord$ on $\Y$
in the standard way (i.e., $y_1\parord y_2$ iff $y_2-y_1\in C$),
and $\mathcal{R}:P\rightrightarrows\X$ is the
feasible region set-valued mapping.
Henceforth $(P,d)$ denotes a metric space, where
perturbation parameters vary, while $(\X,\|\cdot\|)$ and $(\Y,\|\cdot\|)$
denote real Banach spaces.

Fixed $\bar p\in P$, an element $\bar x\in\mathcal{R}(\bar p)$ is said to be
a (global) {\it ideal efficient solution} to the particular problem $(\VOP{\bar p})$
if
\begin{equation}     \label{vle:idefdef}
  f(\bar p,\bar x)\parord f(\bar p,x),\quad\forall x\in\mathcal{R}(\bar p),
\end{equation}
or, equivalently, if
$$
  f(\bar p,\mathcal{R}(\bar p))\subseteq f(\bar p,\bar x)+C.
$$
If the value of the parameter $p$ is subject to perturbations,
making it to vary around the nominal value $\bar p$,
the corresponding problems $(\VOP{p})$ are
expected to admit different ideal efficient solutions, if any,
reflecting changes in the feasible region and in the vector objective function.
The study of the stability behaviour of vector optimization problems
leads therefore to consider the ideal efficient solution mapping $\IESol:P
\rightrightarrows\X$, which is defined by
$$
  \IESol(p)=\{x\in\mathcal{R}(p):\ x \hbox{ ideal efficient solution
  to $(\VOP{p})$}\}.
$$
The analysis of concrete examples gives evidence to the fact that
the behaviour of the mapping $\IESol$ may be bizarre even in the
presence of very amenable data. In the below example, for a
problem with linear (and smoothly perturbed) objective function
and linear (unperturbed) constraints the solution mapping $\IESol$ exhibits
a variety of situations: it alternates isolated solution existence
(meaning no solution for small changes of $p$ around a solvable
problem) with the best form of stability (solution existence
and invariance of the solution set for small changes of $p$).

\begin{example}
Let $P=[0,2\pi]$, $\X=\Y=\R^2$, $C=\R^2_+=\{y=(y_1,y_2)\in\R^2:\
y_1\ge 0,\ y_2\ge 0\}$, let the objective mapping
$f:[0,2\pi]\times\R^2\longrightarrow\R^2$ be given by
$$
   f(p,x)=A(p)x, \quad\hbox{ with }\quad
   A(p)=\left(\begin{array}{cc} \cos (p) & \sin (p) \\
                                -\sin (p) & \cos (p)
                   \end{array}\right),
$$
and let $\mathcal{R}:[0,2\pi]\rightrightarrows\R^2$ be
given by
$$
  \mathcal{R}(p)=T=\{x=(x_1,x_2)\in\R^2:\ x_1\ge 0,\ x_2\ge 0,
  \ x_1+x_2\le 1\},\quad\forall p\in [0,2\pi].
$$
Evidently, the matrix $A(p)$ represents the clockwise rotation
of $\R^2$ of an angle measuring $p$ radians.
By direct inspection of the so defined problem $(\VOP{p})$,
one sees that the associated solution mapping $\IESol:
[0,2\pi]\rightrightarrows\R^2$ results in
$$
  \IESol(p)=\left\{\begin{array}{cl} \{(0,0)\} & \quad\hbox{ if } p=0, \\
                                     \\
                                \{(1,0)\} & \quad\hbox{ if } p\in \left[{\pi\over 2},{3\over 4}\pi\right], \\
                                      \\
                                \{(0,1)\} & \quad\hbox{ if } p\in \left[{5\over 4}\pi,{3\over 2}\pi\right], \\
                                      \\
                                \varnothing & \quad\hbox{ otherwise.}
                   \end{array}\right.
$$
This says that for small changes in the value of
$p\in [0,2\pi]$ near $0$, the corresponding
problems $(\VOP{p})$ have no solution, whereas fixed any $\bar p\in
({\pi\over 2},{3\over 4}\pi)\cup ({5\over 4}\pi,{3\over 2}\pi)$,
for perturbations of the parameter sufficiently near to $\bar p$ the corresponding
problems are still solvable and the solution set stays constant.

It is worth noticing that this parametric optimization problem
admits efficient solutions for every $p\in [0,2\pi]$. Thus, the
present example shows that the geometry of ideal efficiency
can be broken by small perturbations of the parameter more
easily than the one related to mere efficiency.
\end{example}

It is plain to see that the search for ideal efficient solutions to
problems $(\VOP{p})$ can be regarded in fact as a specialization of a more general
class of problems involving set-valued mappings and cones, a kind
of parameterized generalized equations which are referred to as
set-valued inclusions in \cite{Uder21}. More precisely,
given set-valued mappings $\mathcal{R}:P
\rightrightarrows\X$, $F:P\times\X\rightrightarrows\Y$ and a nontrivial
cone $C\subseteq\Y$, these problems require to
$$
  \hbox{ find $x\in\mathcal{R}(p)$ such that }  F(p,x)\subseteq C.
  \leqno (\PSV)
$$
Their solution mapping will be denoted henceforth by $\Solv:P\rightrightarrows\X$,
namely
$$
  \Solv(p)=\{x\in\mathcal{R}(p):\ F(p,x)\subseteq C\}.
$$
By introducing the set-valued mapping $\FRf:P\times\X\rightrightarrows\Y$
defined as
\begin{equation}     \label{eq:FRfdef}
  \FRf(p,x)= f(p,\mathcal{R}(p))-f(p,x),
\end{equation}
it is clear that
$$
  \IESol(p)=\Solv(p).
$$
Set-valued inclusions, in simple as well as in parameterized form, have been
recently studied from several viewpoints in \cite{Cast99,Uder19,Uder20,Uder21,Uder21b}.
The idea underlying the research exposed in the present paper is
that useful insights into the stability behaviour of ideal efficient solutions
can be obtained by
refining in a proper way the study of solution stability of parameterized
set-valued inclusions.
In doing so, it will be also possible to establish some property of the
{\it ideal efficient value mapping} $\ival:\dom\IESol\longrightarrow\Y$ associated to
$(\VOP{p})$, namely the single-valued mapping well defined by
$$
  \ival(p)=f(p,\bar x_p),
$$
where $\bar x_p$ is any element of $\IESol(p)$. Notice that
$\ival(p)$ is well defined even when $\IESol(p)$ contains
more than one element, as it may happen. Indeed, according
to the definition of ideal efficient solution to $(\VOP{p})$,
the relation
$$
 f(p,\bar x_p)\parord f(p,x),\quad\forall x\in\mathcal{R}(p)
$$
must be true for every $\bar x_p\in\IESol(p)$. So, the fact
that $C$ is pointed entails that $f(p,\bar x_p)$ must be
the same value for every $\bar x_p\in\IESol(p)$.

The contents of the paper are arranged as follows. In Section \ref{Sect:2}
a sufficient condition for the existence of ideal efficient solutions
for a problem in the family $(\VOP{p})$, in the case of a fixed value
of $p$, is provided.
In Section \ref{Sect:3} a sufficient condition for the solution mapping
associated to a problem family $(\PSV)$ to be stable is established.
Here the stability behaviour is expressed as Lipschitz lower semicontinuity
for set-valued mappings. An estimate for the related modulus is also
provided.
In Section \ref{Sect:4} the result established in the previous section
finds a specific application in providing sufficient conditions for
the stability of ideal efficient solutions to problems $(\VOP{p})$.
The focus is therefore on Lipschitz lower semicontinuity of $\IESol$,
but, whenever $\IESol$ happens to be single-valued,
such a property qualifies as calmness.
Section \ref{Sect:5} is reserved for concluding remarks
and perspectives.

The main notations in use throughout the paper are basically standard:
$\R$ denotes the real number field and
$\R^n_+$ indicates the nonnegative orthant in the Euclidean space $\R^n$.
In any metric space $(X,d)$,
$\ball{x}{r}$ denotes the closed ball with center $x\in X$ and radius $r\ge 0$,
$\dist{x}{S}$ the distance of $x$ from $S\subseteq X$, with the convention
that $\dist{x}{\varnothing}=+\infty$, and
$\ball{S}{r}=\{x\in X:\ \dist{x}{S}\le r\}$ the $r$-enlargement of $S$.
The symbol $\inte S$ and $\cl S$ indicate the topological interior of $S$
and closure of $S$, respectively.
Given $A,\, B\subseteq X$, the value $\exc{A}{B}=\sup\{\dist{a}{B}:\ a\in A\}$
is the excess of $A$ over $B$.
In any real Banach space $(\X,\|\cdot\|)$, with null vector $\nullv$,
$\Uball=\ball{\nullv}{1}$ stands for the closed unit ball, whereas
$\Usfer$ for the unit sphere.
Given two nonempty subsets $A,\, B\subseteq\X$, their $\ast$-difference
(a.k.a. Pontryagin difference) is defined as
$A\stardif B=\{x\in\X:\ x+B\subseteq A\}$.
The convex hull of a set $A\subseteq\X$ is denoted by $\conv A$.
The space of all $n\times n$ matrices with real entries is
indicated by $\Lin(\R^n)$, the operator norm of $\Lambda\in
\Lin(\R^n)$ by $\|\Lambda\|_\Lin$, and the inverse of $\Lambda$
by $\Lambda^{-1}$.
If $\Phi:\X\rightrightarrows\Y$ denotes a set-valued mapping, its
domain is indicated by $\dom\Phi$.
The acronyms p.h., l.s.c. and u.s.c. stand for positively homogeneous,
lower semicontinuous and upper semicontinuous, respectively.
The meaning of additional symbols will be explained contextually
to their introduction.

\vskip1cm


\section{An existence result without boundedness and continuity}    \label{Sect:2}

This section is a digression from the main theme of the paper.
A basic feature of any stability behaviour of the solution mapping
to a parameterized problem is non-emptiness of its values.
Therefore, before exploring conditions for this phenomenon
to happen, it seems reasonable to spend some words about the
solution existence for a fixed problem within the family
$(\VOP{p})$.
Thus, the present section presents a sufficient condition for the
existence of ideal efficient solutions to the following (geometrically)
constrained vector optimization problem $(\VP)$:
$$
   \Cmin\, f(x) \quad\hbox{ subject to }\quad
   x\in\mathcal{R}, \leqno (\VP)
$$
where $f:X\longrightarrow\Y$, $C\subseteq\Y$ and $\mathcal{R}$
are the problem data. Throughout the current section,
$(X,d)$ stands for a complete metric space, whereas $(\Y,\|\cdot\|)$
denotes a real Banach space.
Such existence condition refines and accomplishes an analogous
result recently proposed (see \cite[Theorem 5.1]{Uder19}), by weakening
several of its hypotheses. Indeed, the continuity of $f$
is replaced by the lower $C$-semicontinuity, while the closedness
of $f(\mathcal{R})$ is dropped out. Besides, an assumption,
given for granted in \cite[Theorem 5.1]{Uder19}, is now explicitly
made, which avoids a pathological, yet possible,
behaviour of $(\VP)$.

Let us recall that,
according to \cite{Luc89}, a mapping $f:X\longrightarrow\Y$ is said to be
$C$-lower semicontinuous (for short, $C$-l.s.c.) at $\bar x\in X$ if for every
$\epsilon>0$ there exists $\delta_\epsilon>0$ such that
\begin{equation}    \label{in:defClsc}
  f(x)\in\ball{f(\bar x)}{\epsilon}+C,\quad\forall x\in\ball{\bar x}{\delta_\epsilon}.
\end{equation}
Of course, whenever $f$ is continuous at $\bar x$, a fortiori it is
$C$-l.s.c. at the same point.
Following a variational approach combined with an analysis
via set-valued inclusions,
the ideal efficient solutions to $(\VP)$ can be singled out
by means of the function $\nu:\mathcal{R}\longrightarrow [0,+\infty]$,
defined by
\begin{equation}    \label{eq:defmeritfun}
  \nu(x)=\exc{f(\mathcal{R})-f(x)}{C}=\exc{f(\mathcal{R})}{f(x)+C}.
\end{equation}
More precisely, since the cone $C$ has been assumed to be closed,
it is clear that
\begin{equation}    \label{eq:solcharnu}
  \IESol=[\nu\le 0]=[\nu=0],
\end{equation}
where $\IESol$ indicates the set of all idel efficient solutions
to $(\VP)$.

The next lemma connects assumptions on the problem data of $(\VP)$
with properties of $\nu$, which will be useful in the sequel.

\begin{lemma}      \label{lem:meritprolsc}
Let $f:X\longrightarrow\Y$ be a mapping, let $C\subseteq\Y$ be
a closed, convex cone and let $\mathcal{R}\subseteq X$ be a
nonempty closed set.

(i) If there exists $x_0\in \mathcal{R}$ such that the set
$[f(\mathcal{R})-f(x_0)]\backslash C$ is bounded, then
$\nu\not\equiv +\infty$.

(ii) If $f$ is $C$-l.s.c. at $\bar x\in \mathcal{R}$, then $\nu$ is
l.s.c. at $\bar x$.
\end{lemma}

\begin{proof}
(i) It suffices to observe that, if $M>0$ is such that
$[f(\mathcal{R})-f(x_0)]\backslash C\subseteq M\Uball$, then it results
in
\begin{eqnarray*}
  \nu(x_0) = \sup_{y\in [f(\mathcal{R})-f(x_0)]\backslash C}\dist{y}{C}
  \le \sup_{y\in M\Uball}\dist{y}{C}
  \le  \sup_{y\in M\Uball}\|y\|=M<+\infty,
\end{eqnarray*}
and hence $\nu\not\equiv +\infty$.

(ii) It is useful to recall that, given two nonempty sets $A,\,
B\subseteq\Y$, and $\epsilon>0$, then
$$
  \exc{A}{B+\epsilon\Uball}\ge\exc{A}{B}-\epsilon.
$$
Indeed, one has
\begin{eqnarray*}
  \exc{A}{B+\epsilon\Uball} &=& \sup_{a\in A}\inf_{b\in B\atop u\in\Uball}\|a-b-\epsilon u\|
  \ge \sup_{a\in A}\inf_{b\in B \atop u\in\Uball} [\|a-b\|-\epsilon\|u\|]  \\
  &=&  \sup_{a\in A} \inf_{b\in B} [\|a-b\|-\epsilon]=\exc{A}{B}-\epsilon.
\end{eqnarray*}
Now, let $(x_n)_n$ be a sequence in $\mathcal{R}$, with $x_n\longrightarrow\bar x$,
as $n\to\infty$. If $\nu(\bar x)=0$, the inequality
$\liminf_{n\to\infty}\nu(x_n)\ge 0=\nu(\bar x)$ trivially holds true,
as $\nu$ takes nonnegative values only.
Assume that $\nu(\bar x)>0$  and take any sequence $(x_n)_n$ in $\mathcal{R}$,
such that $x_n\longrightarrow\bar x$. Fix an arbitrary $\epsilon>0$.
Since $f$ is $C$-l.s.c. at $\bar x$, there exists $\delta_\epsilon>0$ such that
inclusion $(\ref{in:defClsc})$ holds. Since for a proper $\bar n\in\N$, one has
$x_n\in\ball{\bar x}{\delta_\epsilon}$, for every $n\in\N$, $n\ge\bar n$,
then it is $f(x_n)+C\subseteq f(\bar x)+\epsilon\Uball+C$, for every $n\in\N$,
$n\ge\bar n$. Consequently, one obtains
\begin{eqnarray*}
  \nu(x_n) &=& \exc{f(\mathcal{R})}{f(x_n)+C}\ge \exc{f(\mathcal{R})}{f(\bar x)+
  \epsilon\Uball+C}  \\ &\ge &\exc{f(\mathcal{R})}{f(\bar x)+C}-\epsilon
  = \nu(\bar x)-\epsilon, \quad\forall n\in\N,\ n\ge\bar n.
\end{eqnarray*}
The above inequalities imply
$$
  \liminf_{n\to\infty}\nu(x_n)\ge\nu(\bar x)-\epsilon.
$$
The thesis follows by arbitrariness of $\epsilon$. The reader should notice
that such a reasoning works also in the case $\nu(\bar x)=+\infty$.
\hfill$\square$
\end{proof}

In order to formulate the next result, it is to be recalled that, following
\cite[Definition 3.1]{Uder19}, given a set $S\subseteq X$
and a mapping $g:X\longrightarrow\Y$, $g$ is said to be
metrically $C$-increasing on $S$ if there exists a constant $a>1$
such that
\begin{equation}     \label{in:metCincrdef}
  \forall x\in S,\ \forall r>0,\ \exists u\in\ball{x}{r}\cap S:\
  \ball{g(u)}{ar}\subseteq\ball{g(x)+C}{r}.
\end{equation}
The quantity
$$
  \incr{g}{S}=\sup\{a>1:\ \hbox{ inclusion $(\ref{in:metCincrdef})$ holds}\}
$$
is called the {\it exact bound of metric $C$-increase} of $g$ on $S$.
For a discussion of this notion, including examples, related properties
and its connection with the decrease principle of variational analysis,
the reader is referred to \cite{Uder19}.

\begin{theorem}[Ideal efficient solution existence]    \label{thm:iesolexist}
With reference to a problem $(\VP)$, suppose that:

(i) $\mathcal{R}$ is nonempty and closed;

(ii) there exists $x_0\in\mathcal{R}$ such that $[f(\mathcal{R})-f(x_0)]\backslash C$
is a bounded set;

(iii) $f$ is $C$-l.s.c. with respect to the topology induced on $\mathcal{R}$,
at each point of $\mathcal{R}$;

(iv) $-f$ is metrically $C$-increasing on $\mathcal{R}$.

\noindent Then, $\IESol\ne\varnothing$ is closed and the following estimate holds
\begin{equation}     \label{in:erboIESol}
   \dist{x}{\IESol}\le{\nu(x)\over\incr{-f}{\mathcal{R}}},\quad\forall
   x\in \mathcal{R}.
\end{equation}
\end{theorem}

\begin{proof}
The idea is to apply \cite[Theorem 4.2]{Uder19}, after observing that,
as one readily checks by a perusal of its proof, assuming the set-valued
mapping $F:X\rightrightarrows\Y$, $F=f(\mathcal{R})-f$, to be closed-valued
is not required in order for getting the validity of the aforementioned result.

That said, notice that, as a closed subset of a complete metric space,
$\mathcal{R}$ is a complete metric space. In the light of Lemma \ref{lem:meritprolsc},
by hypotheses (ii) and (iii), the function $\nu:\mathcal{R}\longrightarrow [0,+\infty]$
defined as in $(\ref{eq:defmeritfun})$ is l.s.c. on $\mathcal{R}$ and $\nu\not\equiv
+\infty$.
Since according to $(\ref{eq:solcharnu})$ $\IESol=[\nu\le 0]$
is a sublevel set of a l.s.c. function, it is closed.

Now, in order to show that $\IESol\ne\varnothing$ and the error bound in
$(\ref{in:erboIESol})$ holds true, it suffices to apply \cite[Theorem 4.2]{Uder19},
with $X=\mathcal{R}$, $F=f(\mathcal{R})-f$ and $\phi=\nu$,
following the same argument as proposed in \cite[Theorem 5.1]{Uder19}.
In doing so,
notice that the existence of $x_0\in\mathcal{R}$ such that $\nu(x_0)<+\infty$
is guaranteed by hypothesis (ii), whereas the lower semicontinuity of $\nu$
can be
derived directly from the lower $C$-semicontinuity of $f$, instead of
from the lower semicontinuity of $F$.
Besides, the hypothesis (iv) entails the property of metric $C$-increase
on $\mathcal{R}$ for the mapping $f(\mathcal{R})-f$.
\hfill$\square$
\end{proof}

\begin{example}
Let $X=\Y=\R^2$ be endowed with its standard Euclidean space structure,
let $C=\R^2_+$ and let $f:\R^2\longrightarrow\R^2$ be defined by
$$
  f(x)=-x+{\rm e}\sum_{n=0}^{\infty}(n+1)\chi_{(n,n+1]}(\|x\|_\infty),
$$
where ${\rm e}=(1,1)\in\R^2$, $\chi_A$ denotes the characteristic function
of a subset $A\subseteq\R$, and $\|x\|_\infty=\max\{|x_1|,\, |x_2|\}$.
Let the feasible region be $\mathcal{R}=-\R_+{\rm e}=\{x=(x_1,x_2)\in\R^2:\
x_1=x_2\le 0\}$.
One sees from the definition that
\begin{equation}    \label{eq:fRex}
  f(\mathcal{R})=\{(0,0)\}\cup\bigcup_{n=0}^\infty(2n+1,2n+2]{\rm e}.
\end{equation}
This makes it clear that, for the problem $(\VP)$ defined by these data
it is $\IESol=\{(0,0)\}$.
Notice that $f(\mathcal{R})$ fails to be closed, as $(2n+1){\rm e}\not\in
f(\mathcal{R})$, for every $n\in\N$, even though $\mathcal{R}$ is a closed
subset of $\R^2$.
It is readily seen that  $f$ is not continuous at each point of the
form $x=-n{\rm e}\in
\mathcal{R}$, with $n\in\N$. Nonetheless, $f$ turns out to be $\R^2_+$-l.s.c.
at each point of $\mathcal{R}$.
Indeed, fixed any $x_0\in\mathcal{R}$ and $\epsilon\in (0,1)$, it suffices
to take $\delta=\epsilon$ in order to have
\begin{equation}     \label{in:R2+lscf}
   f(x)\in \ball{f(x_0)}{\epsilon}+\R^2_+,\quad\forall x\in\ball{x_0}{\delta}
   \cap\mathcal{R}.
\end{equation}
If $x_0=(0,0)$ this inclusion is evident because $f(\mathcal{R})\subseteq
\R^2_+\subseteq B(f(0),\epsilon)+\R^2_+$. If $x_0\in\bigcup_{n=0}(n,n+1)(-{\rm e})$,
$f$ coincides with the function $x\mapsto -x+(n+1){\rm e}$ in a neighbourhood
in $\mathcal{R}$ of $x_0$ and it is continuous with respect to the topology
induced on $\mathcal{R}$ at $x_0$. If $x_0=-n{\rm e}$, with $n\in\N\backslash
\{0\}$, the inclusion in $(\ref{in:R2+lscf})$ is true because for every
$x\in\ball{x_0}{\epsilon}\cap\mathcal{R}$, with $x_0\parord x$ it is
$f(x)\in\ball{f(x_0)}{\epsilon}$, whereas for every
$x\in\ball{x_0}{\epsilon}\cap\mathcal{R}$, with $x\parord x_0$, $x\ne x_0$,
it results in
\begin{eqnarray*}
   f(x) &=& -x+(n+1){\rm e}\rparord -x_0+(n+1){\rm e}\rparord -x_0+n{\rm e}
   \\
      &=& f(x_0),
\end{eqnarray*}
so
$$
  f(x)\in f(x_0)+\R^2_+\subseteq \ball{f(x_0)}{\epsilon}+\R^2_+.
$$
Let us show that $-f$ is $\R^2_+$-increasing on $\mathcal{R}$.
Fix an arbitrary $x\in\mathcal{R}$ and $r>0$ and set
$$
   u={{\rm e}\over\|{\rm e}\|} \qquad\hbox{ and }\qquad
   z=x+ru\in\ball{x}{r}\cap\mathcal{R}.
$$
Taken $a=2>1$, it is possible to prove that
\begin{equation}     \label{in:R2+incrf}
  -f(z)+ar\Uball\subseteq -f(x)+\R^2_++r\Uball.
\end{equation}
Indeed, since it is $\|x+ru\|_\infty\le\|x\|_\infty$ for every
$x\in\mathcal{R}$, one has
$$
   {\rm e}\sum_{n=0}^{\infty}(n+1)\chi_{(n,n+1]}(\|x+ru\|_\infty)
   \parord
   {\rm e}\sum_{n=0}^{\infty}(n+1)\chi_{(n,n+1]}(\|x\|_\infty)
$$
and hence
$$
  {\rm e}\sum_{n=0}^{\infty}(n+1)\chi_{(n,n+1]}(\|x\|_\infty)\in
  {\rm e}\sum_{n=0}^{\infty}(n+1)\chi_{(n,n+1]}(\|x+ru\|_\infty)+\R^2_+,
$$
wherefrom it follows
$$
  -{\rm e}\sum_{n=0}^{\infty}(n+1)\chi_{(n,n+1]}(\|x+ru\|_\infty)\in
  -{\rm e}\sum_{n=0}^{\infty}(n+1)\chi_{(n,n+1]}(\|x\|_\infty)+\R^2_+,
$$
On the other hand, it is clear that for every $r>0$ it is
\begin{equation}     \label{in:eBalla}
  r{{\rm e}\over\|{\rm e}\|}+ar\Uball=r\ball{{{\rm e}\over\|{\rm e}\|}}
  {a}\subseteq r\Uball+\R^2_+.
\end{equation}
Thus, in the light of the above inclusions, one finds
\begin{eqnarray*}
  -f(z)+ar\Uball &=&(x+ru)-{\rm e}\sum_{n=0}^{\infty}(n+1)\chi_{(n,n+1]}(\|x+ru\|_\infty)+ar\Uball  \\
  &\subseteq & x+r{{\rm e}\over\|{\rm e}\|}-{\rm e}\sum_{n=0}^{\infty}(n+1)\chi_{(n,n+1]}(\|x\|_\infty)
  +\R^2_++ar\Uball  \\
  &\subseteq & f(x)+\left(r{{\rm e}\over\|{\rm e}\|}+ar\Uball\right)+\R^2_+  \\
  &\subseteq & f(x)+r\Uball+\R^2_+,
\end{eqnarray*}
so inclusion $(\ref{in:R2+incrf})$ is satisfied. Moreover, one
can see that $a=2$ is the greatest constant for which inclusion
$(\ref{in:eBalla})$ and hence inclusion $(\ref{in:R2+incrf})$ is true.
Thus, it is $\incr{-f}{\mathcal{R}}=2$.

Thus, since for $x_0=(0,0)$ the set $[f(\mathcal{R})-f(x_0)]\backslash
\R^2_+=\varnothing$ is bounded, for this instance of problem $(\VP)$
Theorem \ref{thm:iesolexist} can be applied. It must be
remarked that the existence of an ideal efficient solution is achieved
in spite of the fact that $\mathcal{R}$ is not bounded, $f(\mathcal{R})$
is not closed and $f$ is not continuous on $\mathcal{R}$.

To accomplish the analysis of the present example, observe that, as
$f(\mathcal{R})$ takes the form in $(\ref{eq:fRex})$ and $(0,0)\in
f(\mathcal{R})$, one readily sees that
\begin{eqnarray*}
   \nu(x) &=& \exc{f(\mathcal{R})-f(x)}{\R^2_+}=
   \exc{f(\mathcal{R})}{f(x)+\R^2_+} =\|f(x)\| \\
   &=& \|x||+\|(n+1){\rm e}\|=\|x\|+\sqrt{2}(n+1),
   \quad\forall x\in\mathcal{R}:\ n<\|x\|_\infty\le n+1.
\end{eqnarray*}
On the other hand, clearly it is $\dist{x}{\IESol}=\|x||$. By taking
into account that, for every $x\in\R^2$, it is $\|x\|\le\sqrt{2}\|x\|_\infty$,
the inequality
$$
  \|x\|_\infty\le n+1
$$
implies
$$
  \|x\|\le\sqrt{2}(n+1).
$$
Thus, one finds
\begin{eqnarray*}
  \dist{x}{\IESol} &=& \|x\|\le{\nu(x)\over\incr{-f}{\mathcal{R}}}
  ={\|x\|+\sqrt{2}(n+1)\over 2}\le{\sqrt{2}(n+1)+\sqrt{2}(n+1)\over 2} \\
  &=&  \sqrt{2}(n+1), \quad\forall x\in\mathcal{R}:\ n<\|x\|_\infty\le n+1,
\end{eqnarray*}
which agrees with the estimate provided in $(\ref{in:erboIESol})$.
\end{example}

Several existence results for ideal efficient solutions can be found
in the literature dedicated to vector optimization. Some of them
demand compactness of the feasible region (see, for instance,
\cite{Oett97}). Other results drop out the boundedness of the feasible
region, while are essentially based on convexity properties of the objective
mapping (see \cite{FloOet01,Flor02}). Theorem \ref{thm:iesolexist} avoids
any form of convexity (remember that $X$ is a metric space), whereas the solution
existence relies on metric completeness, through the property of
metric $C$-increase. Such an approach makes it possible to complement
the qualitative part of the statement (existence) with a quantitative
part (an error bound for the distance from the solution set).

\vskip1cm


\section{Parameterized set-valued inclusions with moving feasible region}  \label{Sect:3}

This section deals with stability properties of the solution mapping
$\Solv:P\rightrightarrows\X$ associated to a parameterized problem $(\PSV)$.
More precisely, a sufficient condition for $\Solv$ to be Lipschitz l.s.c.
at a point of its graph is established.
Recall that, according to \cite{KlaKum02}, a set-valued mapping
$\Phi:P\rightrightarrows X$ between metric spaces is said to be
{\it Lipschitz l.s.c.} at $(\bar p,\bar x)\in\graph\Phi$ if there exist
positive $\delta$ and $\ell$ such that
\begin{equation}    \label{ne:defLiplsc}
  \Phi(p)\cap\ball{\bar x}{\ell d(\bar p,p)}\ne\varnothing,
  \quad\forall p\in\ball{\bar p}{\delta}.
\end{equation}
The value
\begin{equation}     \label{eq:Liplscmoddef}
  \Liplsc{\Phi}{\bar p}{\bar x}=\inf\{\ell>0:\
  \exists\delta>0 \hbox{ for which $(\ref{ne:defLiplsc})$ holds}\}
\end{equation}
is called {\it modulus of Lipschitz lower semicontinuity}
of $\Phi$ at $(\bar p,\bar x)$.

Discussions about this property and its relationships with other
quantitative semicontinuity properties for set-valued mappings
can be found, for instance, in \cite{KlaKum02,Uder21}.
For the purpose of the present analysis, it is relevant to
observe that the requirement in $(\ref{ne:defLiplsc})$ entails
local solvability for problems $(\PSV)$ and nearness to the reference
value $\bar x$ of at least some among the solutions to the perturbed
problems. Not only: the condition postulated in $(\ref{ne:defLiplsc})$
contains a quantitative aspect, in prescribing a nearness
which must be proportional to the parameter variation. The
rate is measured by the modulus of Lipschitz lower semicontinuity.
Historically,
this quantitative aspect motivated the use of the prefix
`Lipschitz' for qualifying such kind of stability
behaviours in the variational analysis literature, to
distinguish them from mere topological properties
(see \cite{DonRoc14,KlaKum02,Mord06,Peno13,RocWet98} and commentaries
therein).

Another property of this kind, which will be employed in
the sequel, is Lipschitz upper semicontinuity: a set-valued mapping
$\Phi:P\rightrightarrows X$ between metric spaces is said to be
{\it Lipschitz u.s.c.} at $\bar p\in\dom\Phi$ if there exist
positive $\delta$ and $\ell$ such that
\begin{equation}    \label{ne:defLipusc}
  \exc{\Phi(p)}{\Phi(\bar p)}\le\ell d(p,\bar p),
  \quad\forall p\in\ball{\bar p}{\delta}.
\end{equation}
The value
$$
  \Lipusc{\Phi}{\bar p}=\inf\{\ell>0:\ \exists\delta>0
  \hbox{ for which $(\ref{ne:defLipusc})$ holds}\}
$$
is called {\it modulus of Lipschitz upper semicontinuity}
of $\Phi$ at $\bar p$.
It is possible to see at once that, whenever $\Phi$ happens to be
single-valued in a neighbourhood of $\bar p$, Lipschitz lower
semicontinuity at $(\bar p,\Phi(\bar p))$ and Lipschitz upper
semicontinuity at $\bar p$ reduce to the same property, as
conditions $(\ref{ne:defLiplsc})$ and $(\ref{ne:defLipusc})$
in this case share the form
\begin{equation}     \label{in:calmdef}
    d(\Phi(p),\Phi(\bar p))\le\ell d(p,\bar p),
    \quad\forall p\in\ball{\bar p}{\delta}.
\end{equation}
If a single-valued mapping $\Phi:P\longrightarrow X$
satisfies inequality $(\ref{in:calmdef})$ for some positive
$\delta$ and $\ell$ it is called {\it calm} at $\bar p$.
In such an event, the value
$$
  \calm{\Phi}{\bar p}=\Liplsc{\Phi}{\bar p}{\Phi(\bar p)}
  =\Lipusc{\Phi}{\bar p}
$$
will be called {\it modulus of calmness} of $\Phi$ at $\bar x$.
When, in particular, $\Phi$ is a single-real-valued function,
the above notion of calmness can be split in its versions
from above and from below. So, $\Phi:P\longrightarrow\R
\cup\{\pm\infty\}$ is said to be {\it calm from above} at
$\bar p\in\dom\Phi$ if there exist positive $\delta$ and
$\ell$ such that
\begin{equation}   \label{in:ucalmdef}
   \Phi(p)-\Phi(\bar p)\le\ell d(p,\bar p),\quad\forall
   p\in\ball{\bar p}{\delta},
\end{equation}
with
$$
  \ucalm{\Phi}{\bar p}=\inf\{\ell>0:\ \exists\delta>0
  \hbox{ for which $(\ref{in:ucalmdef})$ holds}\}.
$$
being the {\it modulus of calmness from above} of $\Phi$
at $\bar p$.

The following standing assumption will be supposed to hold
throughout the current section:

\begin{itemize}

\item[$(\mathcal{A})$] both the set-valued mappings $F$ and
$\mathcal{R}$ take nonempty and closed values (in particular,
$\dom F=P\times\X$ and $\dom\mathcal{R}=P$).

\end{itemize}

In order to develop, through variational methods, a quantitative
stability analysis of the solution mapping associated to $(\PSV)$
it is convenient to introduce the function $\nu_1:P\times\X\longrightarrow
[0,+\infty]$, defined as
\begin{equation}    \label{eq:merfunnu}
  \nu_1(p,x)=\exc{F(p,x)}{C}+\dist{x}{\mathcal{R}(p)},
\end{equation}
which is a kind of merit function providing a functional
characterization of solutions to $(\PSV)$.
In fact, one sees that, for every $p\in P$, it holds
$$
  \Solv(p)=[\nu_1(p,\cdot)=0]=\nu_1(p,\cdot)^{-1}(0).
$$
Together with function $\nu_1$, in what follows it will be
convenient to deal also with the function $\nu_F:P\times\X\longrightarrow
[0,+\infty]$ associated to a set-valued mapping $F:P\times\X
\rightrightarrows\Y$ as being
$$
   \nu_F(p,x)=\exc{F(p,x)}{C}.
$$
Notice that, unlike function $\nu_1$, function $\nu_F$
involves the set-valued mapping $F$ only.

\begin{remark}
The author is aware of the fact that other functions
could be considered to the same purpose in the place of
$\nu_1$, e.g. function $\nu_\infty$ given by
$$
  \nu_\infty(p,x)=\max\{\exc{F(p,x)}{C},
  \dist{x}{\mathcal{R}(p)}\}.
$$
A different choice of merit function does not affect the essence
of the approach and the consequent achievements, resulting only in a change of the
estimates for the involved moduli.
\end{remark}

The variation rate of merit functions such as $\nu_1$
and $\nu_F$ can be measured in a metric space setting
by means of the notion of slope. Recall that, after
\cite{DeMaTo80}, for {\it (strong) slope} of a function $\varphi:X\longrightarrow
\R\cup\{\pm\infty\}$ at $x_0\in\dom\varphi$ the
following value is meant:
\begin{eqnarray*}
  \stsl{\varphi}(x_0)=\left\{
  \begin{array}{ll}
  0, & \hbox{ if $x_0$ is a local minimizer of $\varphi$}, \\
  \displaystyle\limsup_{x\to x_0}{\varphi(x_0)-\varphi(x)\over d(x,x_0)},
  & \hbox{ otherwise.}
  \end{array}\right.
\end{eqnarray*}
A behaviour of the above notion of slope in the presence
of additive perturbations is pointed out in the next
remark, as it will be employed in the sequel.

\begin{remark}[Calm perturbation of the slope]     \label{rem:calmpersl}
Let $\varphi:X\longrightarrow\R\cup\{\pm\infty\}$,
let $\psi:X\longrightarrow\R\cup\{\pm\infty\}$,
and let $x_0\in\dom\varphi\cap\dom\psi$. If $x_0$
is not a local minimizer of $\varphi$, $\psi$
is calm at $x_0$ and $c_\psi>\calm{\psi}{x_0}$, then
$$
  \stsl{(\varphi+\psi)}(x_0)\ge\max\{
  \stsl{\varphi}(x_0)-c_\psi,\, 0\}.
$$
Indeed, according to the definition of strong slope,
one has
$$
  \stsl{(\varphi+\psi)}(x_0)\ge\max\left\{
    \limsup_{x\to x_0}{(\varphi+\psi)(x_0)-(\varphi+\psi)(x)
    \over d(x,x_0)},\, 0 \right\}
$$
and, according to inequality $(\ref{in:calmdef})$, one finds
\begin{eqnarray*}
   \limsup_{x\to x_0}{(\varphi+\psi)(x_0)-(\varphi+\psi)(x)
    \over d(x,x_0)}  &\ge& \limsup_{x\to x_0}{\varphi(x_0)-\varphi(x)
    \over d(x,x_0)}+\liminf_{x\to x_0}{\psi(x_0)-\psi(x)\over d(x,x_0)} \\
    &\ge & \stsl{\varphi}(x_0)-c_\psi.
\end{eqnarray*}
\end{remark}

In the statement of the next result, the following partial version
of the {\it strict outer slope} (see, for instance \cite{FaHeKrOu10}) will
be employed for a function $\varphi:P\times\X\longrightarrow
\R\cup\{\pm\infty\}$ at a point $(p_0,x_0)$:
\begin{eqnarray}    \label{def:psostsl}
  \psostsl{\varphi}(p_0,x_0) &=&\lim_{\epsilon\to 0^+}\inf
    \{\stsl{\varphi(p,\cdot)}(x):\ (p,x)\in \ball{p_0}{\epsilon}\times
    \ball{x_0}{\epsilon}, \\  \nonumber
    & &\hskip4.5cm \varphi(p_0,x_0)<\varphi(p,x)<\varphi(p_0,x_0)+\epsilon\} \\
    \nonumber
    &=&  \liminf_{(p,x)\to(p_0,x_0)\atop \varphi(p,x)\downarrow \varphi(p_0,x_0)}
    \stsl{\varphi(p,\cdot)}(x).
\end{eqnarray}

\begin{proposition}[Lipschitz lower semicontinuity of $\Solv$]   \label{pro:LiplscSolv}
With reference to $(\PSV{p})$, let $\bar p\in P$ and $\bar x
\in\Solv(\bar p)$ be given. Suppose that:

(i) there exists $\delta>0$ such that each mapping $F(p,\cdot):
\X\rightrightarrows\Y$ is l.s.c. on $\X$, for every $p\in
\ball{\bar p}{\delta}$;

(ii) $\mathcal{R}:P\rightrightarrows\X$ is Lipschitz l.s.c. at $(\bar p,\bar x)$
and $F(\cdot,\bar x):P\rightrightarrows\Y$ is Lipschitz u.s.c.
at $\bar p$;

(iii) it holds $\psostsl{\nu_F}(\bar p,\bar x)>1$.

\noindent Then $\Solv$ is Lipschitz l.s.c. at $(\bar p,\bar x)$ and
the following estimate holds
$$
  \Liplsc{\Solv}{\bar p}{\bar x}\le {\Lipusc{F(\cdot,\bar x)}{\bar p}+
  \Liplsc{\mathcal{R}}{\bar p}{\bar x}\over\psostsl{\nu_F}(\bar p,\bar x)-1}.
$$
\end{proposition}

\begin{proof}
Following the same technique as in \cite[Theorem 3.1]{Uder21}, let
us start with showing that, under the current assumptions, the function
$\nu_1:P\times\X\longrightarrow[0,+\infty]$ defined
by $(\ref{eq:merfunnu})$ fulfils the following properties:

\begin{itemize}

\item[$(\wp_1)$] $p\mapsto \nu_1(p,\bar x)$ is calm from above at $\bar p$
and the following estimate holds
$$
  \calmup\nu_1(\cdot,\bar x)\le\Lipusc{F(\cdot,\bar x)}{\bar p}+
  \Liplsc{\mathcal{R}}{\bar p}{\bar x};
$$

\vskip.5cm

\item[$(\wp_2)$] $x\mapsto \nu_1(p,x)$ is l.s.c. on $\X$, for every
$p\in\ball{\bar p}{\delta}$, for some $\delta>0$;

\vskip.5cm

\item[$(\wp_3)$] it holds $\psostsl{\nu_1}(\bar p,\bar x)>0$.

\end{itemize}

As for $(\wp_1)$, by \cite[Lemma 2.4(ii)]{Uder21}, the function
$p\mapsto \exc{F(p,\bar x)}{C}$ is calm from above at $\bar p$
because $F(\cdot,\bar x)$ is Lipschitz u.s.c. at $\bar p$,
with the aforementioned estimate. This means that, for any
$\ell_1>\Lipusc{F(\cdot,\bar x)}{\bar p}$, there exists $\delta_1>0$
such that
$$
  \exc{F(p,\bar x)}{C}\le\ell_1d(p,\bar p),\quad\forall p\in
  \ball{\bar p}{\delta_1}.
$$
On the other hand, by the Lipschitz lower semicontinuity of
$\mathcal{R}$ at $(\bar p,\bar x)$, one can say that for any
$\ell_2>\Liplsc{\mathcal{R}}{\bar p}{\bar x}$ there exists $\delta_2>0$
such that
$$
  \mathcal{R}(p)\cap\ball{\bar x}{\ell_2d(p,\bar p)},
  \quad\forall p\in\ball{\bar p}{\delta_2},
$$
so that $\dist{\bar x}{\mathcal{R}(p)}\le\ell_2d(p,\bar p)$.
Thus, by setting $\delta_0=\min\{\delta_1,\, \delta_2\}$, one obtains
\begin{eqnarray*}
  \nu_1(p,\bar x)-\nu_1(\bar p,\bar x) &\le& \exc{F(p,\bar x}{C}+
  \dist{\bar x}{ \mathcal{R}(p)} \\
  &\le& (\ell_1+\ell_2)d(p,\bar p),\quad\forall p\in\ball{\bar p}{\delta_0}.
\end{eqnarray*}
The last inequality says that the function $\nu_1(\cdot,\bar x)$ is calm
from above at $\bar p$ and, by arbitrariness of $\ell_1$ and $\ell_2$
the estimate in $(\wp_1)$ holds true.

As for $(\wp_2)$, remember that by virtue of the assumption $(\mathcal{A})$
it must be $\mathcal{R}(p)\ne\varnothing$, so that, for every $p\in P$,
each function $x\mapsto\dist{x}{\mathcal{R}(p)}$ is Lipschitz
continuous on $\X$. Besides, by taking $\delta$ as in hypothesis (i),
for every fixed $p\in\ball{\bar p}{\delta}$,
the function $x\mapsto\exc{F(p,x)}{C}$ is l.s.c. on $\X$,
according to \cite[Lemma 2.4(i)]{Uder21}.
Thus the function $x\mapsto\nu_1(p,x)$ turns out to be l.s.c. on $\X$
as a sum of two l.s.c. functions.

As for $(\wp_3)$, according to the hypothesis (iii), fixed $\sigma$
in such a way that $1<\sigma<\psostsl{\nu_F}(\bar p,\bar x)$, there
exists $\delta_\sigma>0$ such that
\begin{equation}      \label{in:stslnuFsigma}
 \stsl{\nu_F(p,\cdot)}(x)>\sigma,\quad\forall (p,x)\in\ball{\bar p}{\delta_\sigma}
 \times\ball{\bar x}{\delta_\sigma}:\ 0<\nu_F(p,x)<\delta_\sigma.
\end{equation}
Fix an arbitrary $(p_0,x_0)\in\ball{\bar p}{\delta_\sigma}
\times\ball{\bar x}{\delta_\sigma}$, with $0<\nu_F(p_0,x_0)<\delta_\sigma$.
The inequality $(\ref{in:stslnuFsigma})$ entails that $x_0$ can not be
a local  minimizer for the function $\nu_F(p_0,\cdot)$. Thus,
since the function $x\mapsto\dist{x}{\mathcal{R}(p_0)}$ is Lipschitz
continuous on $\X$ with constant 1, and hence calm around $x_0$,
it is possible to apply
what has been observed in Remark \ref{rem:calmpersl}, with $\varphi=
\nu_F(p_0,\cdot)$, $\psi=\dist{\cdot}{\mathcal{R}(p_0)}$ and $c_\psi=1$.
Consequently, it holds
$$
  \stsl{[\nu_F(p_0,\cdot)+\dist{\cdot}{\mathcal{R}(p_0)}]}(x_0)
  \ge \stsl{\nu_F(p_0,\cdot)}(x_0)-1\ge\sigma-1>0.
$$
From the last inequality the positivity of $\psostsl{\nu_1}(\bar p,\bar x)$
readily follows.

Now, let us exploit a variational argument to prove the thesis.
By virtue of $(\wp_3)$, there exists $\sigma_0\in (0,1)$ such that
$$
  \psostsl{\nu_1}(\bar p,\bar x)>\sigma_0.
$$
By recalling the definition in $(\ref{def:psostsl})$,
this means that there exists $\eta>0$ such that for every $\epsilon
\in (0,\eta)$ it holds
\begin{equation}    \label{in:stslnu1gesigma0}
  \stsl{\nu_1(p,\cdot)}(x)>\sigma_0,\quad\forall (p,x)\in\ball{\bar p}{\epsilon}
  \times\ball{\bar x}{\epsilon}:\ 0<\nu_1(p,x)<\epsilon.
\end{equation}
Clearly, $\eta$ can be assumed to be smaller that the value of $\delta$
appearing in $(\wp_2)$.
By virtue of property $(\wp_1)$, taken any $\ell>\Lipusc{F(\cdot,\bar x)}{\bar p}+
\Liplsc{\mathcal{R}}{\bar p}{\bar x}$, there exists $\delta_\ell>0$ such that
\begin{equation}     \label{in:nu1ellpbarp}
   \nu_1(p,\bar x)\le\nu_1(\bar p,\bar x)+\ell d(p\bar p)=\ell d(p,\bar p),
   \quad\forall p\in\ball{\bar p}{\delta_\ell}.
\end{equation}
Without loss of generality, one can assume that the inequality in $(\ref{in:nu1ellpbarp})$
holds with
\begin{equation}     \label{in:deltaell}
    0<\delta_\ell<{\sigma_0\eta\over 2(\ell+1)}.
\end{equation}
Notice that, if this is true, one has in particular $\delta_\ell<\eta/2$.

Let us consider the function $\nu_1(p,\cdot):\X\longrightarrow [0,+\infty]$,
where $p$ is arbitrarily fixed in $\ball{\bar p}{\delta_\ell}\backslash\{\bar p\}$.
As it is $\delta_\ell<\eta<\delta$, then by virtue of property $(\wp_2)$, function
$\nu_1(p,\cdot)$ is l.s.c. on $\X$. Moreover, $\nu_1(p,\cdot)$ is obviously bounded
from below and, on account of inequality $(\ref{in:nu1ellpbarp})$, it is
$\nu_1(p,\bar x)<+\infty$ and
$$
 \nu_1(p,\bar x)\le\inf_{x\in\X}\nu_1(p,x)+\ell d(p,\bar p).
$$
These facts enable one to invoke the Ekeland variational principle.
According to it, corresponding to $\lambda=\ell d(p,\bar p)/\sigma_0$,
there exists $x_\lambda\in\X$ such that
\begin{equation}   \label{in:EVP1}
   \nu_1(p,x_\lambda)\le\nu_1(p,\bar x),
\end{equation}
\begin{equation}   \label{in:EVP2}
   d(x_\lambda,\bar x)\le\lambda,
\end{equation}
\begin{equation}   \label{in:EVP3}
   \nu_1(p,x_\lambda)<\nu_1(p,x)+\sigma_0d(x,x_\lambda),
   \quad\forall x\in\X\backslash\{x_\lambda\}.
\end{equation}
In the present context, the validity of the relations $(\ref{in:EVP1})$,
$(\ref{in:EVP2})$ and $(\ref{in:EVP3})$ implies that $\nu_1(p,x_\lambda)=0$.
Indeed, observe that, according to the inequality $(\ref{in:EVP3})$, it is
$$
  {\nu_1(p,x_\lambda)-\nu_1(p,x)\over d(x,x_\lambda)}<\sigma_0,
  \quad\forall x\in\X\backslash\{x_\lambda\},
$$
and hence
\begin{equation}    \label{in:stslnu1sigma0}
   \stsl{\nu_1(p,\cdot)}(x_\lambda)=\lim_{r\to 0^+}
   \sup_{x\in\ball{x_\lambda}{r}\backslash\{x_\lambda\}}
   {\nu_1(p,x_\lambda)-\nu_1(p,x)\over d(x,x_\lambda)}\le\sigma_0.
\end{equation}
On the other hand, by recalling that $d(p,\bar p)\le\delta_\ell<\eta/2$,
on account of inequalities $(\ref{in:EVP2})$ and $(\ref{in:deltaell})$
one finds
$$
   d(x_\lambda,\bar x)\le {\ell d(p,\bar p)\over\sigma_0}\le
   {\ell\over\sigma_0}\delta_\ell<{\eta\over 2}.
$$
Besides, by combining inequalities $(\ref{in:nu1ellpbarp})$,
$(\ref{in:deltaell})$ and $(\ref{in:EVP1})$, one obtains
$$
  \nu_1(p,x_\lambda)\le\ell\delta_\ell<{\eta\over 2}.
$$
Thus, if it were $\nu_1(p,x_\lambda)>0$, in the light of inequality
$(\ref{in:stslnu1sigma0})$ one would find inequality $(\ref{in:stslnu1gesigma0})$
contradicted for $\epsilon=\eta/2$.

The fact that $\nu_1(p,x_\lambda)=0$ means that
$$
  \exc{F(p,x_\lambda)}{C}=0 \qquad\hbox{ and }\qquad
  \dist{x_\lambda}{\mathcal{R}(p)}=0,
$$
so, as $\mathcal{R}(p)$ and $C$ are closed sets,
$$
  F(p,x_\lambda)\subseteq C \qquad\hbox{ and }\qquad
  x_\lambda\in\mathcal{R}(p),
$$
namely $x_\lambda\in\Solv(p)$. Since it is $d(x_\lambda,\bar x)\le
\ell  d(p,\bar p)/\sigma_0$, as a consequence one has
$$
  \Solv(p)\cap\ball{\bar x}{{\ell d(p,\bar p)\over\sigma_0}}
  \ne\varnothing.
$$
By arbitrariness of $p\in\ball{\bar p}{\delta_\ell}\backslash\{\bar p\}$,
this allows one to say that $\Solv$ is Lipschitz l.s.c. at
$(\bar p,\bar x)$ and
$$
  \Liplsc{\Solv}{\bar p}{\bar x}\le{\ell\over\sigma_0}.
$$
As the last inequality remains true for every $\ell>\Lipusc{F(\cdot,\bar x)}{\bar p}+
\Liplsc{\mathcal{R}}{\bar p}{\bar x}$ and for every $\sigma_0<
\psostsl{\nu_1}(\bar p,\bar x)$, then also the estimate in the thesis
must hold true. This completes the proof.
\hfill$\square$
\end{proof}

From the proof of Proposition \ref{pro:LiplscSolv} it should be
evident that such a result embeds Theorem 3.1 in \cite{Uder21},
which provides a sufficient condition for Lipschitz lower semicontinuity
in the special case with $\mathcal{R}$ being given by $\mathcal{R}(p)=X$,
for every $p\in P$. Notice that, in such an event,
$\Liplsc{\mathcal{R}}{\bar p}{\bar x}=0$ while, for every $p\in P$,
the function $x\mapsto\dist{x}{\mathcal{R}(p)}$ vanishes.
The condition in hypothesis (iii) can therefore be replaced with
the mere positivity of $\psostsl{\nu_F}(\bar p,\bar x)$, as $\nu_1$
reduces to $\nu_F$.

\vskip1cm


\section{Stability conditions for ideal efficiency}   \label{Sect:4}

In the present section, with the aim of deriving a stability
condition for ideal efficiency, the general condition for
the Lipschitz lower semicontinuity of the solution mapping
associated to a parameterized set-valued inclusion presented
in Section \ref{Sect:3}
will be adapted to the specific context of vector optimization problems.
In such a setting, the set-valued mapping $F$ appearing in problems
$(\PSV)$ takes the special form introduced in $(\ref{eq:FRfdef})$.
While in Proposition \ref{pro:LiplscSolv} several assumptions
are directly made on $F$, inasmuch as in the context of $(\PSV)$
such a mapping appears among the problem data as an independent
one, the definition of $\FRf$ involves several elementary
data such as $\mathcal{R}$ and $f$. This fact requires a further
work aimed at singling out reasonable conditions, which can guarantee
the aforementioned assumptions  be satisfied.

\begin{remark} \label{rem:closedness}
Under conditions making each set-valued mapping $\FRf(p,\cdot):\X\rightrightarrows\Y$
l.s.c. on $\X$, for $p\in P$, the mapping $\IESol:P\rightrightarrows\X$ turns out
to be closed (possibly, empty) valued.
\end{remark}

Throughtout the current section, the following assumption will
be supposed to hold

\begin{itemize}

\item[$(\tilde{\mathcal{A}})$] $\dom\mathcal{R}=P$.

\end{itemize}

\begin{lemma}[Lower semicontinuity of $\FRf$]    \label{lem:lscFRf}
Let $p\in\ P$ and let the mapping $f(p,\cdot):\X\longrightarrow\Y$
be continuous on $\X$. Then, the set-valued mapping
$\FRf(p,\cdot):\X\rightrightarrows\Y$ defined as in
$(\ref{eq:FRfdef})$ is l.s.c. on $\X$.
\end{lemma}

\begin{proof}
Observe that, as a consequence of assumption $(\tilde{\mathcal{A}})$,
it is $\dom\FRf=P\times X$.
Fix $x_0\in\X$ and take an arbitrary open subset $O$ of $\Y$,
with $\FRf(p,x_0)\cap O\ne\varnothing$. According to the
definition of $\FRf$, this means
$$
  [f(p,\mathcal{R}(p))-f(p,x_0)]\cap O\ne\varnothing,
$$
so there exists $y_0\in f(p,\mathcal{R}(p))$ such that
$y_0-f(p,x_0)\in O$. By openness of $O$, there exists $\epsilon>0$
such that $\ball{y_0-f(p,x_0)}{\epsilon}\subseteq O$. Thus, since
the function $f(p,\cdot)$ is continuous at $x_0$, there exists
$\delta_\epsilon>0$ such that
$$
   f(p,x)\in\ball{f(p,x_0)}{\epsilon},\quad\forall
   x\in\ball{x_0}{\delta_\epsilon},
$$
wherefrom it follows
$$
  y_0-f(p,x)\in\ball{y_0-f(p,x_0)}{\epsilon},\quad\forall
   x\in\ball{x_0}{\delta_\epsilon}.
$$
Consequently, one finds
$$
  y_0-f(p,x)\in\FRf(p,x)\cap O\ne\varnothing,
  \quad\forall  x\in\ball{x_0}{\delta_\epsilon},
$$
what shows that $\FRf(p,\cdot)$ is l.s.c. at $x_0$, thereby
completing the proof.
\hfill$\square$
\end{proof}

\begin{lemma}      \label{lem:LipusccompG}
Let $f:P\times\X\longrightarrow\Y$ be a mapping, let $\mathcal{R}:
P\rightrightarrows\X$ be a set-valued mapping with $\dom\mathcal{R}=P$,
and let $\bar p\in\ P$. Suppose that:

\begin{itemize}

\item[(i)] $f$ is Lipschitz continuous with constant $\ell_f$ on
$P\times\X$;

\item[(ii)] $\mathcal{R}$ is Lipschitz u.s.c. at $\bar p$.

\end{itemize}

\noindent Then, the set-valued mapping $G:P\rightrightarrows\Y$,
defined by $G(p)=f(p,\mathcal{R}(p))$, is Lipschitz u.s.c. at
$\bar p$ and the following estimate holds
\begin{equation}     \label{in:Lipusccompest}
    \Lipusc{G}{\bar p}\le\ell_f[1+\Lipusc{\mathcal{R}}{\bar p}].
\end{equation}
\end{lemma}

\begin{proof}
By hypothesis (ii), fixed $\ell_\mathcal{R}>\Lipusc{\mathcal{R}}{\bar p}$
there exists $\delta>0$ such that
\begin{equation}    \label{in:LipuscRbarp}
   \exc{\mathcal{R}(p)}{\mathcal{R}(\bar p)}\le\ell_\mathcal{R} d(p,\bar p),
   \quad\forall p\in\ball{\bar p}{\delta}.
\end{equation}
Take an arbitrary $p\in\ball{\bar p}{\delta}$ and $x\in\mathcal{R}(p)$.
By virtue of the Lipschitz continuity of $f$, one obtains
\begin{eqnarray*}
    \dist{f(p,x)}{G(\bar p)} &=& \inf_{z\in\mathcal{R}(\bar p)} \|f(p,x)-f(\bar p,z)\|
    \le\inf_{z\in\mathcal{R}(\bar p)}\ell_f[d(p,\bar p)+d(x,z)]  \\
    &=& \ell_f\left[d(p,\bar p)+\inf_{z\in\mathcal{R}(\bar p)}d(x,z)\right]
    = \ell_f[d(p,\bar p)+\dist{x}{\mathcal{R}(\bar p)}].
\end{eqnarray*}
As from inequality $(\ref{in:LipuscRbarp})$ one has for every
$x\in\mathcal{R}(p)$
$$
  \dist{x}{\mathcal{R}(\bar p)}\le\ell_\mathcal{R} d(p,\bar p),
   \quad\forall p\in\ball{\bar p}{\delta},
$$
then the last estimate gives
$$
  \dist{f(p,x)}{G(\bar p)}\le \ell_f[1+\ell_\mathcal{R}]d(p,\bar p),
   \quad\forall p\in\ball{\bar p}{\delta}.
$$
By arbitrariness of $x$ in $\mathcal{R}(p)$, what obtained implies
$$
  G(p)=f(p,\mathcal{R}(p))\subseteq\ball{G(\bar p)}{\ell_f[1+\ell_\mathcal{R}]d(p,\bar p)},
  \quad\forall p\in\ball{\bar p}{\delta}.
$$
This shows that $G$ is Lipschitz u.s.c. at $\bar p$ with
$\Lipusc{G}{\bar p}\le\ell_f[1+\ell_\mathcal{R}]$.
The arbitrariness of $\ell_\mathcal{R}>\Lipusc{\mathcal{R}}{\bar p}$
enables one to achieve the estimate in $(\ref{in:Lipusccompest})$.
\hfill$\square$
\end{proof}

The next lemma establishes a stability behaviour of the Lipschitz
upper semicontinuity property under additive calm perturbations,
which turns out to be useful in the present approach.

\begin{lemma}      \label{lem:Lipusccalmadd}
Let $G:P\rightrightarrows\Y$ be a set-valued mapping, let
$h:P\longrightarrow\Y$ be a given single-valued mapping and
let $\bar p\in P$.
If $G$ is Lipschitz u.s.c. at $\bar p$ and $h$ is calm at $\bar p$,
then $G+h$ is Lipschitz u.s.c. at $\bar p$ and the following
estimate holds
\begin{equation}   \label{in:Lipuscmodestsum}
   \Lipusc{(G+h)}{\bar p}\le\Lipusc{G}{\bar p}+\calm{h}{\bar p}.
\end{equation}
\end{lemma}

\begin{proof}
It suffices to observe that, since any distance induced by a norm
is invariant under translations, one has
\begin{eqnarray*}
    \exc{G(p)+h(p)}{G(\bar p)+h(\bar p)} &=&\sup_{y\in G(p)}
    \dist{y+h(p)}{G(\bar p)+h(\bar p)}    \\
    &=& \sup_{y\in G(p)} \dist{y}{G(\bar p)+h(\bar p)-h(p)} \\
    &\le & \exc{G(p)}{G(\bar p)} +\|h(p)-h(\bar p)\|.
\end{eqnarray*}
The estimate in $(\ref{in:Lipuscmodestsum})$ is a straightforward
consequence of the above inequality, the definitions of modulus
of Lipschitz upper semicontinuity and of modulus of calmness.
\hfill$\square$
\end{proof}

Conditions ensuring the behaviour of $\psostsl{\nu_F}(\bar p,\bar x)$
to fit the requirement in hypothesis (iii) of Proposition \ref{pro:LiplscSolv} will be
expressed in terms of generalized derivatives. Recall that, following
\cite{Robi91}, a mapping $f:\X\longrightarrow\Y$ is said to be
{\it Bouligand differentiable} at $x_0\in\X$ if there exists a
continuous p.h. mapping $\Bder{f}{x_0;\cdot}:\X\longrightarrow\Y$
such that
$$
  \lim_{x\to x_0}{f(x)-f(x_0)-\Bder{f}{x_0;x-x_0}
  \over\|x-x_0\|}=0.
$$
In such an event, the mapping $v\mapsto\Bder{f}{x_0;v}$ is
called Bouligand derivative of $f$ at $x_0$. It is clear
that such a differentiability notion actually generalizes
the Fr\'echet smoothness: whenever $f$ is Fr\'echet differentiable
at $x_0$, with Fr\'echet derivative $\Fder f(x_0)$,
$f$ is also Bouligand differentiable at the same point with
$\Bder{f}{x_0;\cdot}=\Fder f(x_0)$.

Before stating the next remark, it is proper to recall that,
after \cite{Ioff81}, a p.h. set-valued mapping $H(x_0;\cdot):
\X\rightrightarrows\Y$ is said to be an {\it outer prederivative}
of $G:\X\rightrightarrows\Y$ at $x_0\in\X$ if for every $\epsilon>0$
there exists $\delta>0$ such that
$$
  G(x)\subseteq G(x_0)+H(x_0;x-x_0)+\epsilon\|x-x_0\|\Uball,
  \quad\forall x\in\ball{x_0}{\delta}.
$$
For more details on this nonsmooth analysis tool the reader
may refer to \cite{Ioff81,Pang11}.

\begin{remark}     \label{rem:Bdifoutpreder}
Let $f:P\times\X\longrightarrow\Y$ be a mapping, let $p$ be fixed in $P$
and $x_0\in\X$. If the mapping $x\mapsto f(p,\cdot)$ is Bouligand differentiable
at $x_0$ with Bouligand derivative $\Bder{f(p,\cdot)}{x_0}$,
then the set-valued mapping $x\leadsto\FRf(p,x)$ admits as
an outer prederivative
at $x_0$ the mapping $v\leadsto \{-\Bder{f(p,\cdot)}{x_0}(v)\}$.
Indeed, fixed any $\epsilon>0$, the Bouligand differentiability
of $f(p,\cdot)$ at $x_0$ ensures the existance of $\delta_\epsilon>0$
such that
$$
  f(p,x)\in f(p,x_0)+\Bder{f(p,\cdot)}{x_0;x-x_0}+
  \epsilon\|x-x_0\|\Uball,\quad\forall x\in\ball{x_0}{\delta_\epsilon}.
$$
This inclusion implies
\begin{eqnarray*}
  \FRf(p,x) &=& f(p,\mathcal{R}(p))-f(p,x) \\
    &\subseteq & f(p,\mathcal{R}(p))-f(p,x_0)-\Bder{f(p,\cdot)}{x_0;x-x_0}+
  \epsilon\|x-x_0\|\Uball  \\
  &=& \FRf(p,x_0)-\Bder{f(p,\cdot)}{x_0;x-x_0}+
  \epsilon\|x-x_0\|\Uball,
  \quad\forall x\in\ball{x_0}{\delta_\epsilon} .
\end{eqnarray*}
\end{remark}

The next technical lemma provides a below estimate for the
slope of the function $\nu_{\FRf}(p,\cdot):\X\longrightarrow [0+\infty]$,
defined by $\nu_{\FRf}(p,x)=\exc{\FRf(p,x)}{C}$, in terms
of `strict negativity' (with respect to the partial ordering
$\parord$) of the values taken by the first-order approximation
of $f(p,\cdot)$.

\begin{lemma}     \label{lem:stslBdersigma}
With reference to a problem $(\VOP{p})$, let $p$ be fixed in
$P$ and let $x_0\not\in\IESol(p)$. Suppose that:

\begin{itemize}

\item[(i)] $f(p,\cdot):\X\longrightarrow\Y$ is continuous
on $\X$;

\item[(ii)] $f(p,\cdot)$ is Bouligand differentiable at $x_0$;

\item[(iii)] there exist $\sigma>1$ and $u\in\Usfer$ such that
$\ball{\Bder{f(p,\cdot)}{x_0;u}}{\sigma}\subseteq -C$.

\end{itemize}

\noindent Then, it holds
\begin{equation}     \label{in:stslFRfpos}
  \stsl{\nu_{\FRf}(p,\cdot)}(x_0)\ge\sigma.
\end{equation}
\end{lemma}

\begin{proof}
By virtue of hypothesis (i) and Lemma \ref{lem:lscFRf}, the
set-valued mapping $\FRf(p,\cdot)$ is l.s.c. on $\X$, so,
in particular, l.s.c. at $x_0$. According with what has been noticed
in Remark \ref{rem:Bdifoutpreder}, $\FRf(p,\cdot)$ admits the
set-valued mapping $v\leadsto \{-\Bder{f(p,\cdot)}{x_0}(v)\}$
as an outer prederivative at $x_0$, owing to hypothesis (ii).

Now, if $\sigma$ and $u\in\Usfer$ are as in hypothesis (iii), one
has
$$
  -\Bder{f(p,\cdot)}{x_0;u}+\sigma\Uball\subseteq C
$$
and hence
$$
  \sup_{v\in\Usfer}|C\stardif \{-\Bder{f(p,\cdot)}{x_0}(v)\}|
  \ge\sigma,
$$
where $|S|=\sup\{r>0:\ r\Uball\subseteq S\}$. In the light of
\cite[Proposition 2.5]{Uder21b}, the last inequality implies
the estimate in $(\ref{in:stslFRfpos})$, thereby completing
the proof.
\hfill$\square$
\end{proof}

With the above elements, one is in a position to establish
the following result about stability of ideal efficient
solutions to $(\VOP{p})$.

\begin{theorem}[Lipschitz lower semicontinuity of $\IESol$]   \label{thm:LiplscIE}
With reference to a $(\VOP{p})$, let $\bar p\in P$ and
$\bar x\in\IESol(\bar p)$ be given. Suppose that:

\begin{itemize}

\item[(i)] $f:P\times\X\longrightarrow\Y$ is Lipschitz
continuous on $P\times\X$ with constant $\ell_f$;

\item[(ii)] $\mathcal{R}$ is Lipschitz u.s.c. at $\bar p$
and Lipschitz l.s.c. at $(\bar p,\bar x)$;

\item[(iii)] there exists $\delta_0>0$ such that $f(p,\cdot)$
is Bouligand differentiable on $\ball{\bar x}{\delta_0}$, for
each $p\in\ball{\bar p}{\delta_0}$;

\item[(iv)] there exist $\delta\in (0,\delta_0)$ and $\sigma>1$ such that
for every $(p,x)\in [\ball{\bar p}{\delta}\times\ball{\bar x}{\delta}] \\
\backslash\graph\IESol$ there is $u\in\Usfer$ such that
\begin{equation}     \label{in:Bdersigmapos}
  \Bder{f(p,\cdot)}{x;u}+\sigma\Uball\subset -C.
\end{equation}
\end{itemize}

\noindent Then, $\IESol$ is Lipschitz l.s.c. at $(\bar p,\bar x)$
and the following estimate holds
\begin{equation}     \label{in:LiplscestIESol}
  \Liplsc{\IESol}{\bar p}{\bar x}\le {\ell_f[2+\Lipusc{\mathcal{R}}{\bar p}]+
  \Liplsc{\mathcal{R}}{\bar p}{\bar x}\over\sigma-1}.
\end{equation}
\end{theorem}

\begin{proof}
The proof consists in showing that, under the current assumptions, it is
possible to apply Proposition \ref{pro:LiplscSolv}, with $F=\FRf$. To do so,
let us start with observing, since each mapping $x\mapsto f(p,x)$ is
continuous on $\X$, for every $p\in P$, as a consequence of hypothesis (i),
then on account of Lemma \ref{lem:lscFRf} each set-valued mapping
$\FRf(p,\cdot)$ is l.s.c. on $\X$, for every $p\in P$.
This shows that hypothesis (i) of Proposition \ref{pro:LiplscSolv}
is fulfilled.

Moreover, by virtue of hypothesis (i) and (ii), Lemma \ref{lem:LipusccompG}
ensures that the set-valued mapping $p\leadsto f(p,\mathcal{R}(p))$
is Lipschitz u.s.c. at $\bar p$, with $\Lipusc{f(\cdot,\mathcal{R}(\cdot))}{\bar p}
\le \ell_f[1+\Lipusc{\mathcal{R}}{\bar p}]$. Since for any fixed $x\in\X$ the mapping
$p\mapsto f(p,x)$ is calm at $\bar p$, again as a consequence of
hypothesis (i), with constant $\ell_f>\calm{f(\cdot,x)}{\bar p}$,
then Lemma \ref{lem:Lipusccalmadd} enables one to say that $\FRf$
is Lipscitz u.s.c. at $\bar p$, with
$$
  \Lipusc{\FRf}{\bar p}\le\ell_f[1+\Lipusc{\mathcal{R}}{\bar p}]
  +\ell_f.
$$
This shows that all the requirements in the hypothesis (ii) of
Proposition \ref{pro:LiplscSolv} are fulfilled under the
assumptions made.

It remains to show that also hypothesis (iii) of Proposition
\ref{pro:LiplscSolv} is fulfilled. This can be done by
applying Lemma \ref{lem:stslBdersigma}. Remembering the definition
of partial strict outer slope, one has to prove the existence
of $\epsilon>0$ such that
$$
  \stsl{\nu_{\FRf}(p,\cdot)}(x)>1,\quad\forall (p,x)\in
  \ball{\bar p}{\epsilon}\times\ball{\bar x}{\epsilon},\quad
  0<\nu_{\FRf}(p,x)<\epsilon.
$$
So, taking $\epsilon\in (0,\delta)$, where $\delta>0$ is as in hypothesis (iv),
and an arbitrary $(p,x)\in\ball{\bar p}{\epsilon}\times\ball{\bar x}{\epsilon}$,
one has that, according to hypothesis (iii), if it is $\nu_{\FRf}(p,x)>0$
then $(p,x)\not\in\graph\IESol$ and therefore, by hypothesis (iv)
there exists $u\in\Usfer$ such that inclusion $(\ref{in:Bdersigmapos})$
holds. In turn this inclusion, on account of Lemma \ref{lem:stslBdersigma},
implies that $\stsl{\nu_{\FRf}(p,\cdot)}(x)\ge\sigma>1$.

Thus the thesis follows by taking into account that, in the
current setting, $\IESol=\Solv$.
This completes the proof.
\hfill$\square$
\end{proof}

Hypothesis (ii) in Theorem \ref{thm:LiplscIE} refers to a certain
stability behaviour of $\mathcal{R}$. In concrete problems, this
set-valued mapping is defined by a large variety of constraint
systems. For many of them, in the last decades adequate conditions
ensuring the needed stability behaviour have been developed
within variational analysis (see \cite[Chapter 4.D]{DonRoc14},
\cite{KlaKum02},\cite[Chapter 4.3]{Mord06} and references therein).

The stability behaviour of $\IESol$ established by Theorem \ref{thm:LiplscIE}
has a remarkable consequence on the stability of ideal
efficient values, which can be formulated through the
mapping $\ival:P\longrightarrow\Y$.

\begin{corollary}[Calmness of $\ival$]
Under the same hypotheses as in Theorem \ref{thm:LiplscIE} the mapping
$\ival:P\longrightarrow\Y$ is calm at $\bar p$ and it holds
$$
   \calm{\ival}{\bar p}\le {\ell_f^2[2+\Lipusc{\mathcal{R}}{\bar p}]+
  \ell_f(\Liplsc{\mathcal{R}}{\bar p}{\bar x}+1)\over\sigma-1}.
$$
\end{corollary}

\begin{proof}
By Theorem \ref{thm:LiplscIE} $\IESol$ is Lipschitz l.s.c. at
$(\bar p,\bar x)$, with the related modulus estimate. So,
if taking an arbitrary $\ell>\Liplsc{\IESol}{\bar p}{\bar x}$,
there exists $\zeta_\ell>0$ such that for any $p\in
\ball{\bar p}{\zeta_\ell}$ an element $x_p$ must belong to
$\IESol(p)$ with the property that $d(x_p,\bar x)\le\ell d(p,\bar p)$.
Consequently, it results in
\begin{eqnarray*}
   |\ival(p)-\ival(\bar p)| &=& |f(p,x_p)-f(\bar p,\bar x)|
   \le\ell_f[d(p,\bar p)+\|x_p-\bar x\|]   \\
   &\le & \ell_f[1+\ell]d(p,\bar p), \quad\forall p
   \in\ball{\bar p}{\zeta_\ell}.
\end{eqnarray*}
This says that $\ival$ is calm at $\bar p$.
By arbitrariness of $\ell>\Liplsc{\IESol}{\bar p}{\bar x}$,
to obtain the estimate complementing the thesis, it suffices
to recall the inequality in $(\ref{in:LiplscestIESol})$.
\hfill$\square$
\end{proof}

\begin{example}
Let $P=[0,+\infty)$, $\X=\Y=\R^2$, $C=\R^2_+$, with $f:[0,+\infty)
\times\R^2\longrightarrow\R^2$ given by
$$
  f(p,x)=(2\arctan x_2,-2\arctan x_1),
$$
and $\mathcal{R}:[0,+\infty)\rightrightarrows\R^2$ given by
$$
  \mathcal{R}=\{x\in\R^2:\ x_1\ge 0,\ x_2\ge 0,\ x_1+x_2
  \le\beta(p)\},
$$
where $\beta:[0,+\infty)\longrightarrow[0,+\infty)$ is
a function with $\beta(0)=0$ and calm from above at $0$.
Take $\bar p=0$ and $\bar x=(0,0)$.

In order to find the ideal efficient solutions to
the corresponding $(\VOP{p})$, it is convenient to observe
first that $\IESol(0)=\{(0,0)\}$ and that, for every $y=(y_1,y_2)
\in f(p,\mathcal{R}(p))$, with $p\in [0,+\infty)$, according
to the definition of $f$ and $\mathcal{R}(p)$, one has
$$
  y_1\ge 0 \qquad\hbox{ and }\qquad y_2\ge -2\arctan\beta(p).
$$
In other terms, $(\beta(p),0)\in\mathcal{R}(p)$ and
$$
  f(p,\mathcal{R}(p))\subseteq f(p,(\beta(p),0))+\R^2_+=
  (0,-2\arctan\beta(p))+\R^2_+,
$$
which means that $(\beta(p),0)\in\IESol(p)$, for every $p\in [0,+\infty)$.
Besides, since the vector $(0,-2\arctan\beta(p))$ can be the only ideal efficient
element of the set $f(p,\mathcal{R}(p))$ and the function $x\mapsto f(p,x)$
is injective, one can state that
$$
   \IESol(p)=\{(\beta(p),0)\},\quad\forall p\in [0,+\infty).
$$
Thus, since for any $c_\beta>\ucalm{\beta}{0}$ there exists $\delta>0$
such that it holds
$$
  \dist{(0,0)}{\IESol(p)}=\beta(p)\le c_\beta p,
  \quad\forall p\in [0,\delta],
$$
it is possible to deduce that $\IESol$ is Lipschitz l.s.c.
(actually, also Lipschitz u.s.c. and hence calm) at $(0,(0,0))$,
with
\begin{equation}    \label{in:LiplscIESolexdir}
  \Liplsc{\IESol}{0}{(0,0)}\le\ucalm{\beta}{0}.
\end{equation}
In order to test the application of Theorem \ref{thm:LiplscIE}
in this concrete case, let us start with noticing that,
since $f(p,\cdot)$ is (Fr\'echet) differentiable on
$\R^2$ and the linear mapping $\Fder f(p,\cdot)(x):\R^2\longrightarrow\R^2$
can be represented by the Jacobian matrix
$$
  \Fder f(p,\cdot)(x)=
  \left(\begin{array}{cc} 0 & \displaystyle{2\over 1+x_2^2} \\
                          -\displaystyle{2\over 1+x_1^2} & 0
                   \end{array}\right),
$$
with
\begin{eqnarray*}
   \|\Fder f(p,\cdot)(x)\|_\Lin &=&  \sup_{u\in\Usfer}
   \left\|
   \left(\begin{array}{cc} 0 & \displaystyle{2\over 1+x_2^2} \\
                          -\displaystyle{2\over 1+x_1^2} & 0
                   \end{array}\right)
   \left(\begin{array}{c} u_1 \\
                          u_2
                   \end{array}\right)
   \right\|  \\
   & & \\
   &=& \sup_{u\in\Usfer} \left\|
   \left(\begin{array}{cc} \displaystyle{2u_2\over 1+x_2^2}, &
   -\displaystyle{2u_1\over 1+x_1^2}
   \end{array}\right)
   \right\| \le
   \sqrt{{4\over (1+x_1^2)^2}+{4\over (1+x_2^2)^2}}  \\
   &\le& 2\sqrt{2},\quad  \forall x=(x_1,x_2)\in\R^2,
\end{eqnarray*}
then $f$ turns out to be Lipschitz continuous on $[0,+\infty)
\times\R^2$, with constant $\ell_f=2\sqrt{2}$. Since it is
$$
  \exc{\mathcal{R}(p)}{\mathcal{R}(0)}=\|(\beta(p),0)\|=\beta(p)
  \le c_\beta p,\quad\forall p\in [0,\delta],
$$
it is true that $\mathcal{R}$ is Lipschitz u.s.c. at $0$, with
$\Lipusc{\mathcal{R}}{0}\le c_\beta$. Moreover, as it is
$$
  \mathcal{R}(0)=\{(0,0)\}\subseteq\mathcal{R}(p),
  \quad\forall p\in [0,+\infty),
$$
one sees that for every $\ell>0$ is holds
$$
  \mathcal{R}(p)\cap\ball{(0,0)}{\ell|p|}\ne\varnothing,
  \quad\forall p\in [0,+\infty),
$$
what says that $\mathcal{R}$ is also Lipschitz l.s.c. at
$(0,(0,0))$ and $\Liplsc{\mathcal{R}}{0}{(0,0)}=0$.

Now, take an arbitrary $x\in\ball{(0,0)}{\delta}\backslash
\{(0,0)\}$, with $\delta$ fixed in such a way that
$0<\delta<\sqrt{\sqrt{2}-1}$, and set
$$
  \sigma={\sqrt{2}\over 1+\delta^2}.
$$
Notice that $\sigma>1$, because $\delta<\sqrt{\sqrt{2}-1}$.
Taking $u=(1/\sqrt{2},-1/\sqrt{2})\in\Usfer$, one finds
$$
  \Fder f(p,\cdot)(x)u=
  \left(\begin{array}{cc} -\displaystyle{\sqrt{2}\over 1+x_2^2}, &
   -\displaystyle{\sqrt{2}\over 1+x_1^2}\end{array}\right),
$$
whence it follows
$$
 \dist{\Fder f(p,\cdot)(x)u}{\R^2\backslash(-\inte\R^2_+)}=
 \min\left\{{\sqrt{2}\over 1+x_1^2},\, {\sqrt{2}\over 1+x_2^2}\right\}
 \ge {\sqrt{2}\over 1+\delta^2}.
$$
Consequently, it is true that
$$
  \Fder f(p,\cdot)(x)u+\sigma\Uball\subseteq -\R^2_+, \quad\forall
  x\in\ball{(0,0)}{\delta}\backslash\{(0,0)\}.
$$
This shows that also hypothesis (iii) of Theorem \ref{thm:LiplscIE}
is fulfilled. In the case under consideration, the estimate in
$(\ref{in:LiplscestIESol})$ becomes
$$
  \Liplsc{\IESol}{0}{(0,0)}\le {2\sqrt{2}[2+c_\beta]\over
  \displaystyle{\sqrt{2}\over 1+\delta^2}-1},
$$
which is consistent with (even though, less accurate than) the estimate in
$(\ref{in:LiplscIESolexdir})$, obtained by direct inspection of
$\IESol$. Indeed, one sees that
$$
  \lim_{\delta\to 0^+}  {2\sqrt{2}[2+c_\beta]\over
  \displaystyle{\sqrt{2}\over 1+\delta^2}-1} =
  {4\sqrt{2}+2\sqrt{2}c_\beta\over \sqrt{2}-1}>c_\beta>
  \ucalm{\beta}{0}
$$
(whereas
$$
  \lim_{\delta\to{\sqrt{\sqrt{2}-1}\ }^-}{2\sqrt{2}[2+c_\beta]\over
  \displaystyle{\sqrt{2}\over 1+\delta^2}-1}=+\infty>\ucalm{\beta}{0}\ ).
$$

\end{example}

The above example suggests that, whenever $f$ is one-to-one, $\IESol$ is
single-valued and this fact automatically enhances the Lipschitz lower semicontinuity
property to calmness. It is well known that a sufficient condition for
a Lipschitz (possibly, nonsmooth) mapping $f$ between finite-dimensional spaces
to be a homeomorphism can be expressed in terms of Clarke's
generalized Jacobian (see \cite{Pour82}).
Let $\partial^\circ f(p,\cdot)(x_0)$ denote the {\it Clarke's generalized
Jacobian} of $f(p,\cdot):\R^n\longrightarrow\R^n$ at $x_0\in\R^n$,
i.e. the set
\begin{eqnarray*}
  \partial^\circ f(p,\cdot)(x_0) &=& \conv\left\{\Lambda\in\Lin(\R^n):\
  \exists (x_k)_k,\ x_k\in \Diff{f(p,\cdot)},\ x_k\to x_0,  \right. \\
  & & \left. \Fder f(p,\cdot)(x_k)\longrightarrow\Lambda \hbox{ as } k\to\infty \right\},
\end{eqnarray*}
where $\Diff{f(p,\cdot)}$ indicates the set of points at which the function
$x\mapsto f(p,x)$ is (Fr\'echet) differentiable (the Rademacher theorem
ensures that such a set is a Lebesgue full measure subset of $\R^n)$.

\begin{corollary}
Under the same hypotheses as in Theorem \ref{thm:LiplscIE}, suppose that
$\X=\Y=\R^n$ and

\begin{itemize}

\item[(v)] for every $p\in P$ there exists $\gamma_p>0$ such that, for
every $x\in\R^n$, every $\Lambda\in\partial^\circ f(p,\cdot)(x)$ is invertible
and
$$
  \sup_{x\in\R^n}\sup_{\Lambda\in\partial^\circ f(p,\cdot)(x)}
  \|\Lambda^{-1}\|_\Lin\le\gamma_p.
$$

\end{itemize}

\noindent Then, $\IESol$ is single-valued and calm at $\bar p$, with
$$
  \calm{\IESol}{\bar p}\le {\ell_f[2+\Lipusc{\mathcal{R}}{\bar p}]+
  \Liplsc{\mathcal{R}}{\bar p}{\bar x}\over\sigma-1}.
$$
\end{corollary}

\begin{proof}

Fix an arbitrary $p\in P$. The additional hypothesis (v) enables one
to apply the Lipschitzian Hadamard theorem in \cite{Pour82}.
According to it, the mapping $f(p,\cdot):\R^n\longrightarrow\R^n$
is one-to-one on $\R^n$. Consequently, since it is
$$
   \IESol(p)=f^{-1}(p,\cdot)(\ival(p))\cap\mathcal{R}(p),
$$
the mapping $\IESol$ must be single-valued. As already remarked,
in such a circumstance Lipschitz lower semicontinuity and calmness
collapse to the same property. So the thesis becomes a consequence
of Theorem \ref{thm:LiplscIE}.
\hfill$\square$
\end{proof}


\section{Conclusions}   \label{Sect:5}

Evidences show that ideal efficiency has a delicate geometry.
The findings of the present paper demonstrate that the analysis
of the solution stability for parameterized set-valued inclusions
can afford useful insights into the behaviour of ideal efficient solutions
to vector optimization problems subject to perturbations,
from both the qualitative and the quantitative viewpoint.
The study has focused on the Lipschitz lower semicontinuity property
for the ideal efficient solution mapping, but it is reasonable to expect
that other quantitative stability properties widely considered in
variational analysis (such as Lipschitz upper
semicontinuity, calmness and the Aubin property) can be fruitfully
investigated by the same approach, via set-valued inclusions.
While the analysis of parameterized set-valued inclusions has been conducted
in a rather abstract setting, the related achievements have been
subsequently applied in a more structured context, where the employment
of well-known generalized derivatives ensures applicability
of results to a large class of problems.
The choice made in this part of the work leaves open the possibility
to refine the stability conditions here obtained by means of other,
more sophisticated, tools of nonsmooth analysis.

\vskip1cm


\end{document}